\documentclass[12pt,reqno,a4paper]{article}

\usepackage{amsmath,amsthm,amssymb}
\usepackage{authblk}
\usepackage{color}
\usepackage{soul,xcolor}

\usepackage{enumitem}

\usepackage[mathscr]{euscript}
\usepackage[T1]{fontenc}
\usepackage{graphics}
\usepackage{fullpage}
\usepackage[backref,colorlinks,linkcolor=blue]{hyperref}
\usepackage{mathtools}
\usepackage{multicol}

\usepackage[colorinlistoftodos,prependcaption]{todonotes}
\usepackage[all]{xy}

\newtheorem{theorem}{Theorem}[section]
\newtheorem{corollary}[theorem]{Corollary}
\newtheorem{lemma}[theorem]{Lemma}
\newtheorem{proposition}[theorem]{Proposition}

\newtheorem*{conjecture}{Conjecture}
\theoremstyle{definition}
\newtheorem{definition}[theorem]{Definition}

\newtheorem{example}[theorem]{Example}
\theoremstyle{remark}
\newtheorem{remark}[theorem]{Remark}

\numberwithin{equation}{section}
\numberwithin{figure}{section}

\newcommand{\B}[1]{{\mathbf #1}}
\newcommand{\C}[1]{{\mathscr #1}}

\newcommand{\AL}[1]{\textcolor{red}{#1}}

\begin{document}

\title{Torsion in asymptotic cones of free groups}
\author{Jarek K\k{e}dra}
\affil{
University of Aberdeen and University of Szczecin\\
{\tt kedra@abdn.ac.uk}
}
\author{Assaf Libman}
\affil{
University of Aberdeen\\
{\tt a.libman@abdn.ac.uk}
}
\author{Zipei Nie}
\affil{
University of Illinois\\
{\tt znie@illinois.edu}
}

\maketitle

\begin{abstract}
We prove that the asymptotic cone of any non-abelian free group equipped with the bi-invariant word metric has elements of order $2$.
\end{abstract}

\section{Introduction}\label{S:intro}

\paragraph{Bi-invariant metrics.}
Recall that a metric $d$ on a group $G$ is called {\em bi-invariant} if for any $f,g,h\in G$
$$
d(fg,fh) = d(gf,hf) = d(g,h).
$$
In other words, the actions of $G$ on itself by both the left and the right multiplication are by isometries of $d$.


\paragraph{Asymptotic cones.}
Let $\omega$ be a non-principal ultrafilter on the set of positive integers.
It allows us to choose in a canonical way the limit of a convergent subsequence in any bounded sequence of real numbers.

The asymptotic cone of a metric space $(X,d)$ with respect to $\omega$ (and the scaling sequence $s_n=n$) is defined as follows, see \cite{zbMATH06866242} for more details.
Define a pseudo-metric on the set of sequences $(x_n)$ in $X$ for which $\{\tfrac{d(x_n,x_1)}{n}: n \geq 1 \}$ is bounded by 
\[
d_\omega((x_n),(y_n)) = \lim_\omega \frac{d(x_n,y_n)}{n}.
\]
The equivalence classes under the relation $(x_n) \sim (y_n) \iff d_\omega((x_n),(y_n))=0$ are the elements of ${\rm Cone}_\omega(X,d)$ which is equipped with the metric $d_\omega([x_n],[y_n])=d_\omega((x_n),(y_n))$.

It is well known that the asymptotic cone ${\rm Cone}_{\omega}(G,d)$ of a group $G$ equipped with a bi-invariant metric $d$ is itself a complete metric group where the group structure is defined by $[g_n] \cdot [h_n] = [g_n h_n]$.
Moreover the metric $d_{\omega}$  is bi-invariant \cite{zbMATH06563655}; see Section \ref{S:preliminaries} for more details.

The purpose of this note is to prove the following theorem.

\begin{theorem}\label{T:main}
Let $\B F$ be a non-abelian free group equipped with the bi-invariant word metric $d_{\overline{S}}$ associated with a set $S$ of free generators.
Then the asymptotic cone ${\rm Cone}_{\omega}(\B F,d_{\overline{S}})$ has non-trivial 
elements of order $2$ for any choice of non-principal ultrafilter $\omega$.
\end{theorem}

\paragraph{Strategy of the proof.}
We prove Theorem \ref{T:main} by an explicit construction of a sequence of
elements $f_n\in \B F$ such that $\|f_n\|=d_{\overline{S}}(f_n,1)$ grows linearly with $n$
and the norms $\|f_n^2\|$ of their squares grow sublinearly with $n$.  The
proof has three steps. 
The first and most substantial step is the proof of the following theorem.
Throughout this paper, given a set $S$  we write $\B F(S)$ for the free group freely generated by $S$. For an integer $d \geq 0$ we write $\B F(d)$ for the free group on $d$ generators.

\begin{theorem}[See Theorem \ref{T:bounds}]\label{thm:Nie intro}
For any $n\geq 2$ there exists an element 
$g_n \in \B F(4n-3)$ such that
\begin{equation}
\begin{aligned}
n^2-n+1 \leq &\|g_n\| \leq 4n^2-8n+5\\
2n \leq &\|g_n^2\| \leq 8n-4.
\end{aligned}
\label{Eq:main}
\end{equation}
\end{theorem}
The construction of $g_n$'s and the proof of the above inequalities (Theorem
\ref{T:bounds}) is due to
the third  author \cite[Theorem 5.1]{1412.0101}.
This step is highly non-trivial.
Since the paper~\cite{1412.0101} has not been published, we present the 
details in Section \ref{S:Finfty}.
In the second step we define homomorphisms  
$\Psi_n\colon \B F({4n-3})\to \B F(2)$
such that
$$
\|\Psi_n(g_n)\| = \|g_n\| \qquad \text{and} \qquad \|\Psi_n(g_n^2)\| = \|g_n^2\|.
$$
Finally, we modify the sequence $(\Psi_n(g_n))$ to obtain a non-trivial element
of order $2$ in the asymptotic cone 
${\rm Cone}_{\omega}(\B F(2))$.
Since 
$\B F(2)$ 
embeds isometrically into the free group of any rank, this
completes the proof.

\paragraph{Comparison with the commutator length.}
An analogous theorem is known for the commutator length in free groups.
It is a result of Kharlampovich and Myasnikov \cite[Theorem 3 and 4]{zbMATH01744024}, 
from which follows the existence of non-trivial $2$-torsion in the asymptotic cone
${\rm Cone}_{\omega}([\B F,\B F],{\rm cl})$, where ${\rm cl}(g)$ denotes
the commutator length of $g\in [\B F,\B F]$. 
Notice that their theorem is not constructive.
An attempt to construct such elements explicitly was done by Bartholdi et al.
\cite[Theorem D]{zbMATH07861316}, where they constructed  
$g\in [\B F(2),\B F(2)]$ of word length $64$ such that
$$
{\rm cl}(g)=3\qquad\text{and}\qquad {\rm cl}(g^2)=2.
$$

It follows from inequalities \eqref{Eq:main} that an element $g=g_9\in \B F(33)$
satisfies 
$$
\|g^2\| \leq 68 < 73\leq \|g\|.
$$
However, it follows from the construction that this element is represented by a
word of length $37^{16}$. 
There are examples of shorter elements $g$ in the free group such that $\| g^2 \| < \|g \|$.
The third author used a computer-aided search \cite{cancellation-norm-search} to find the following element $g \in \B F(\{a,b\}) \cong \B F(2)$ of length $33$:
\[
g = aba^{-1}b^{-1}a^{-1}b^{-2}ab^{2}ab^{-1}a^{-1}b^{-1}aba^{-1}b^{-2}a^{-1}b^{2}ab^{-1}a^{-1}b^{2}abab^{-1}a^{-1}b^{-1}.
\]
It satisfies $\| g\|=5$ but $\| g^2\|=4$; see Example \ref{E:gg2} for more details.

It remains unknown whether commutator length can decrease under taking cubes, that is, whether ${\rm cl}(g^3)<{\rm cl}(g)$ can occur \cite[Section 3.3]{zbMATH07861316}. In contrast, we make the following conjecture for the conjugation-invariant word norm.
\begin{conjecture}
For every non-trivial $g\in \B F$,
\[
\|g^3\|>\|g\|.
\]
\end{conjecture}

\begin{remark}
The sequence $(g_n)$ constructed in Section \ref{S:Finfty} can be easily modified
to be contained in the commutator subgroup.
They may be potential explicit examples for the result of Kharlampovich-Myasnikov. 
However, we don't know how to estimate their commutator length.
\end{remark}

\paragraph{Acknowledgements.}
For the purpose of open access, the author has applied a Creative Commons
Attribution (CC BY) licence to any Author Accepted Manuscript version
arising from this submission.

\section{Preliminaries}\label{S:preliminaries}

\paragraph{Norms and metrics on groups.}\label{PA:metrics}
Let $G$ be a group. A function $\nu\colon G\to \B R$ is called a {\em norm} if
it satisfies the following axioms for all $g,h\in G$.
\begin{enumerate}
\item $\nu(g)\geq 0$ and $\nu(g)=0$ if and only if $g=1$;
\item $\nu(g^{-1}) = \nu(g)$;
\item $\nu(gh) \leq \nu(g)+\nu(h)$.
\end{enumerate}
If, in addition, 
\begin{enumerate}
\item[4.]
$\nu(g^{-1}hg) = \nu(h)$
\end{enumerate}
then $\nu$ is called {\em conjugation-invariant}.
The associated metric $d_{\nu}$ is defined by
$$
d_{\nu}(g,h) = \nu(g^{-1}h).
$$
It is left-invariant in the sense that $d_{\nu}(fg,fh) = d_{\nu}(g,h)$. That is, the
left action of $G$ on itself is by isometries. A norm $\nu$ is conjugation-invariant
if and only if $d_{\nu}$ is bi-invariant:
$$
d_{\nu}(fg,fh) = d_{\nu}(gf,hf) = d_{\nu}(g,h),
$$
for all $f,g,h\in G$.

\paragraph{Word metrics.}\label{PA:word}
If $S\subseteq G$ is a symmetric generating set then the associated
{\em word norm} is defined by
$$
|g|_S = \min\{n\in \mathbb N\ |\ g=s_1\dots s_n,\ s_i\in S\}.
$$
Being symmetric means $S^{-1}=S$.
The associated metric is called the {\em word metric} and is denoted by $d_S$.
If $S$ is invariant under conjugations, that is, $S=g^{-1}Sg$ for every $g\in G$,
then the associated word norm is conjugation-invariant and the associated
word metric is bi-invariant.

Let $\B F= \B F(S)$ be a free group on a set $S$ containing at least two elements. 
Let $\overline{S} = \bigcup_{g\in \B F}g^{-1}S^{\pm 1}g$ be the union of the
conjugacy classes of all generators $s\in S$ and their inverses.
This paper is concerned with the conjugation-invariant word norm and
the bi-invariant word metric on $\B F$ associated with $\overline{S}$. 
Throughout the paper we will use the following {\bf notation}:
\begin{itemize}
\item $\ell(w)$ - the word length of a word $w$ in the alphabet $S\cup S^{-1}$;
\item $|g|=|g|_S$ - the word norm associated with $S$; 
this is the length of the unique reduced word in the alphabet $S^{\pm 1}$ representing $g$. Notice that $|w| \leq \ell(w)$ for any (not-necessarily reduced) word.
\item $\|g\|= |g|_{\overline{S}}$ - the conjugation-invariant word norm associated with $\overline{S}$;
\item $d(g,h) = \|g^{-1}h\|$ - the associated bi-invariant metric.
\end{itemize}
Let $v,w$ be words in the alphabet $S\cup S^{-1}$. The word $v=t_1t_2\dots t_m$ is called
a {\em subsequence} of $w=s_1s_2 \dots s_n$ if $t_j = s_{i_j}$ for some 
subsequence $(i_j)_{j=1}^m$ of the sequence $(1,2,\ldots,n)$. A subsequence
consisting of letters with consecutive indices is called a {\em subword}.

\paragraph{Cancellation length and folding.}\label{PA:cancel}
Let $w$ be a word in the alphabet $S\cup S^{-1}$. A {\em cancellation sequence}
for $w$ is a subsequence $c$ of $w$ such that deleting $c$ from $w$ yields a word
representing the trivial element in the group $\B F(S)$.  For example,
$bab^{-1}a^{-1}$ is a cancellation sequence in
$$
[a,b]^3 = 
a{\textcolor{blue}{\underline{b}}}
a^{-1}b^{-1}\textcolor{blue}{\underline{a}}
ba^{-1}{\textcolor{blue}{\underline{b^{-1}}}}
ab\textcolor{blue}{\underline{a^{-1}}}
b^{-1}.
$$
The {\em cancellation length} of $w$ is the smallest word length $\ell(c)$ of a cancellation sequence of $w$.

\begin{example}\label{E:gg2}
The following are cancellation sequences of shortest length of the elements $g$ and $g^2$ mentioned in the
introduction.
\[
g = aba^{-1}b^{-1}a^{-1}b^{-2}ab\textcolor{blue}{\underline{ba}}b^{-1}a^{-1}\textcolor{blue}{\underline{b^{-1}}}aba^{-1}b^{-2}\textcolor{blue}{\underline{a^{-1}}}b^{2}ab^{-1}a^{-1}b^{2}abab^{-1}a^{-1}\textcolor{blue}{\underline{b^{-1}}},
\]
\begin{align*}
g^2 = &aba^{-1}b^{-1}a^{-1}b^{-2}abbab^{-1}a^{-1}\textcolor{blue}{\underline{b^{-1}}}aba^{-1}b^{-2}a^{-1}b^{2}ab^{-1}a^{-1}b^{2}abab^{-1}a^{-1}\textcolor{blue}{\underline {b^{-1}}}\\
&aba^{-1}b^{-1}a^{-1}b^{-2}abbab^{-1}a^{-1}b^{-1}aba^{-1}b^{-2}\textcolor{blue}{\underline{a^{-1}}}b^{2}ab^{-1}a^{-1}b^{2}ab\textcolor{blue}{\underline{a}}b^{-1}a^{-1}b^{-1}.
\end{align*}
\end{example}

The following lemma is straightforward and it was observed in 
\cite[Proposition 2.5]{zbMATH06532523}.
\begin{lemma}\label{L:cancel}
Let $g\in \B F(S)$ be represented by a word $w$, not necessarily reduced. 
Then $\|g\|$ is equal to the cancellation length of $w$. In particular, cancellation lengths of any two words representing~$g$ are equal.
\qed
\end{lemma}

\begin{remark}
Since the length of the shortest cancellation sequence does not depend on the choice of a word $w$ representing $g$, the cancellation length provides a powerful tool for computations
involving conjugation-invariant norm on free groups. This was demonstrated,
for example, in the proof of the rationality of the stabilisation of
the conjugation-invariant norm on free groups \cite{zbMATH07843754}.
\end{remark}

Let $w=s_1\dots s_n$ be a word and let $c$ be a cancellation sequence.
Let $u=s_{i_1}\dots s_{i_{k}}$ be the word obtained from $w$ by deleting
the letters in $c$. It represents a trivial element and
its trivialisation is by successive removal of pairs of letters $ss^{-1}$
or $s^{-1}s$, called {\em contractions}. Keeping track
of the contractions defines a partition $\C F$ of 
$\{i_1,i_2,\ldots,i_{k}\}\subseteq \{1,2,\ldots,n\}$ into
two-element subsets with the property that if $\{i,j\},\{k,\ell\}\in \C F$
then the intervals $[i,j],[k,\ell]\subseteq \mathbb R$ are either disjoint or one is contained in the other. Notice that if $\{i,j\}\in \C F$ then $s_i = s_j^{-1}$.
For example, if $c=a^2$ is a cancellation sequence in 
$w = a^2b\textcolor{blue}{\underline{a}}b^{-2}\textcolor{blue}{\underline{a}}ba^{-2}$, 
then 
$u = a^2bb^{-2}ba^{-2}$ and 
$$
\{3,5\}\sqcup \{6,8\}\sqcup \{2,9\}\sqcup \{1,10\}=\{1,2,3,5,6,8,9,10\}
$$
is the partition corresponding to contractions $bb^{-1}$, $b^{-1}b$, $aa^{-1}$
and $aa^{-1}$. Such a partition $\C F$ is called a {\em folding} of $u$.

\begin{definition}
A {\em cancellation system} of a word $w$ is a pair $(c,\C F)$, where
$c$ is a cancellation sequence of $w$ and $\C F$ is a folding for
the word $u$ obtained by removing the letters of $c$ from $w$.
\end{definition}

\begin{lemma}\label{L:folding}
Let $w=s_1\dots s_n$ be a word and let $(c,\C F)$ be a cancellation system.
Suppose $\{i,j\}\in \C F$, where $i<j$, and let
\begin{align*}
w_L&=s_1\dots s_{i-1}\\
w_0&=s_i\dots s_{j}\\
w_R&=s_{j+1}\dots s_{n}.
\end{align*}
Let $c_L,c_0,c_R$ be subwords of $c$ consisting of letters in
$w_L,w_0$ and $w_R$, respectively.
Then $c_0$ is a cancellation sequence of $w_0$ and $c_Lc_R$ is a cancellation
sequence of $w_Lw_R$. In particular, if $c$ is minimal then
$$
\|w\| \geq \|w_0\| + \|w_Lw_R\|.
$$
\end{lemma}
\begin{proof}
If $\{k,l\} \in \C F$ is such that $i<k<j$ then also $i<l<j$ by definition
of a folding. This implies that $c_0$ is a cancellation sequence of $w_0$.
Consequently, removing $c_0$ from $c$ yields a cancellation system
$(c_Lc_R,\C F')$ of $w_Lw_R$. In particular, $c_Lc_R$ is a cancellation
sequence for $w_Lw_R$. If $c$ is minimal then
$$
\|w\| = \ell(c) = \ell(c_0)+\ell(c_Lc_R) \geq \|w_0\|+\|w_Lw_R\|,
$$
which finishes the proof.
\end{proof}

\paragraph{Asymptotic cones.}\label{PA:acones}
An ultrafilter $\omega$ on the set $\mathbb N$ of non-negative integers 
is a finitely additive probability measure on the set of all subsets of $\mathbb N$ 
with values in the set $\{0,1\}$. 
It is called {\em non-principal} if every finite
subset of $\mathbb N$ has measure zero; see \cite[Section 7]{zbMATH06866242}
for more details. In what follows we implicitly assume that
all ultrafilters we consider are non-principal.
For every bounded sequence $(a_n)$ of
real numbers the ultrafilter $\omega$ defines its ultralimit 
$$
\lim_{\omega}a_n = a
$$
by the property that for every $\epsilon > 0$ 
$$
\omega\{n\in \mathbb N\ |\ |a_n-a|<\epsilon\} = 1.
$$

Let $(X,d)$ be a pointed metric space and let $*\in X$ be the basepoint.
Consider the set of sequences $(x_n)$ in $X$ such that 
\begin{equation}
\lim_{\omega} \frac{d(*,x_n)}{n} < \infty.
\label{Eq:linear}
\end{equation}
Two such sequences $(x_n)$ and $(y_n)$ are called equivalent if 
$$
\lim_{\omega} \frac{d(x_n,y_n)}{n}=0.
$$
The {\em asymptotic cone} ${\rm Cone}_{\omega}(X,d)$ of $(X,d)$ 
associated with an ultrafilter $\omega$ and the scaling sequence 
$s_n=n$
is the set of equivalence classes of such sequences together
with the metric defined by
\begin{equation}
d_{\omega}([x_n],[y_n]) = \lim_{\omega}\frac{d(x_n,y_n)}{n}.
\label{Eq:cone-metric}
\end{equation}
It is well known that $({\rm Cone}_{\omega}(X,d),d_{\omega})$ is a complete metric
space \cite[Proposition 7.44]{zbMATH06866242}. 
Asymptotic cones are in general defined for
sequences of metric spaces with other scaling sequences $(s_n)$, i.e. sequences such that $s_n \to \infty$, which 
changes the denominators in \eqref{Eq:linear} and \eqref{Eq:cone-metric} to $s_n$ in place of $n$;
see \cite[Section 7]{zbMATH06866242} for more details. The following straightforward
lemma was first observed by Calegari and Zhuang \cite{zbMATH06563655}.

\begin{lemma}\label{L:cone-group}
If $(G,d)$ is a group equipped with a bi-invariant metric then for any
non-principal ultrafilter $\omega$ the asymptotic cone ${\rm Cone}_{\omega}(G,d)$
is a metric group with the metric~$d_{\omega}$ being bi-invariant.\qed
\end{lemma}

\paragraph{Notational conventions.}\label{PA:notation}
In complicated formulas we will sometimes use a dot $\cdot$ to denote the multiplication.
For example, $g\cdot h$ instead of $gh$. Similarly, in the notation
for a composition of functions we will sometimes skip the circle $\circ$ to make
formulas less cumbersome. For example, 
$$
\pi_{\{j,j+1\}}\varphi^{(j+1)}\varphi_{j+1}\varphi_j(h)\qquad
\text{instead of}\qquad
\left(\pi_{\{j,j+1\}}\circ \varphi^{(j+1)}\circ \varphi_{j+1}\circ\varphi_j\right)(h).
$$
The following are our conventions
for the conjugation and the commutator.
\begin{itemize}
\item Conjugation:
\[
h^g = g^{-1}hg.
\]
If the conjugating element is not relevant, 
we will write $h^{\bullet}$ instead of $h^g$. 
Thus, $f=h^{\bullet}$ means
that $f$ is equal to {\em some} conjugate of $h$.

\item Commutator: 
\[
[g,h]= ghg^{-1}h^{-1} = g\cdot (g^{-1})^{h^{-1}} = h^{g^{-1}}\cdot h^{-1}
=g\cdot(g^{-1})^{\bullet}=h^{\bullet}\cdot h^{-1}.
\]
\end{itemize}

\section{Proof of Theorem \ref{thm:Nie intro}}\label{S:Finfty}

\subsection{The construction of the elements  $g_n \in \B F_{\infty}$}\label{SS:construction}
Let $\mathbb N$ denote the set of non-negative integers.
Let 
$$
S = \{x_0,y_0,x_1,y_1,\ldots\} = \{x_i,y_i\ |\ i\in \mathbb N\}
$$
and let $\B F_{\infty} = \B F(S)$ be the free group generated by $S$.
Throughout the paper, $s_i\in S$ denotes either $x_i$ or $y_i$.

Throughout this section we fix once and for all an integer
\[
n\geq 2.
\]
For any integer $k \geq 1$ set
\begin{equation}
T_k = (x_ky_k)^n\left( x_k^{-1}y_k^{-1} \right)^n 
= [x_k,(y_kx_k)^n] = [y_k^{-1},(y_kx_k)^n],
\label{Eq:2-commutators}
\end{equation}
an element of $\B F(\{x_k,y_k\})$.
Define a homomorphism 
\[
\varphi_k\colon \B F_{\infty}\to \B F_{\infty}
\]
by specifying its values on generators as follows
\begin{equation}
\varphi_k(s_i) =
\begin{cases}
s_i\cdot T_k & \text{ if either $i$ or $k$ is even;}\\
T_k^{-1}\cdot s_i & \text{ if both $i$ and $k$ are odd.}
\end{cases}
\label{Eq:function-fi}
\end{equation}
Observe that $\ell(T_k)=4n$, so for any $g \in \B F_\infty$ 
\begin{equation}\label{Eq:phi_k ell}
| \varphi_k(g)| \leq |g|+|g| \cdot 4n = (4n+1)|g|.
\end{equation}
Also, if $g \in \B F(A)$ for some $A \subseteq \{x_0,y_0, x_1,y_1, \dots\}$ then $\varphi_k(g) \in \B F(A \cup \{x_k,y_k\})$.
For any $k=1,2,\ldots,2n-1$ set
\[
\varphi^{(k)} = \varphi_{2n}\circ\varphi_{2n-1}\circ \dots \circ \varphi_{k+1},
\]
and set
\[
\varphi^{(2n)}={\rm Id}.
\]
Finally, set
\begin{equation}
g_n = \varphi^{(2)}(x_0)=
\left(\varphi_{2n}\circ \varphi_{2n-1}\circ \dots \circ \varphi_3\right)(x_0).
\label{Eq:gn}
\end{equation}
Observe that $g_n \in \B F(\{x_0, x_3,y_3,\dots,x_{2n},y_{2n}\}) \cong \B F(4n-3)$ and, by a repeated application of \eqref{Eq:phi_k ell}
\begin{equation}\label{Eq:ell g_n}
|g_n | \leq (4n+1)^{2n-2}.
\end{equation}
The main result of this section is the following estimates,
whose proof occupies the rest of this section.

\begin{theorem}\label{T:bounds}
Let $n\geq 2$ and let $g_n\in\langle x_0,y_0, \dots, x_{2n},y_{2n} \rangle \leq \B F_\infty$ be as above. Then
\begin{align}
n^2-n+1 \leq \|g_n\|   &\leq 4n^2-8n+5 \label{Eq:norm-gn}\\
2n\leq \|g_n^2\| &\leq 8n-4 \label{Eq:norm-gn2}.
\end{align}
\end{theorem}

\subsection{The upper bounds}

\paragraph{Properties of the homomorphisms $\varphi_k$.}
We will often need the value of $\varphi_k$ on the product of two generators
and we start with a simple lemma which deals with the necessary cases.
Recall that $g^{\bullet}$ denotes some conjugate of $g$. 
We will use
the following simple property of the ``bullet notation''
$$
gh = h^{\bullet}g = hg^{\bullet}. 
$$

\begin{lemma}\label{L:25}
Let $g,h\in \B F_{\infty}$ and let $s_i,s_j\in S$ be generators. Then
\begin{enumerate}[label=(\arabic*)]
\item \label{L:25:gT}
For any $k\geq 1$,
$gT_k = gy_k^{-1}\cdot y_k^{\bullet}$;
\item \label{L:25:gThT}
For any $k\geq 1$,
$gT_khT_k 
= gy_k^{-1}\cdot(hx_k)^{\bullet}\cdot(y_kx_k^{-1})^{\bullet}$.

\item \label{L:25:phi ss even}
If $k\geq 1$ is even then $\varphi_k(s_i^{-1}s_j)=(s_i^{-1}s_j)^{\bullet}$.
\item 
\label{L:25:phi ss odd}
If $i,j\in \mathbb N$ have the same parity then
\begin{enumerate}[label=(\alph*)]
\item \label{L:25:phi ss odd:lower}
$\varphi_k(s_is_j^{-1}) = (s_is_j^{-1})^{\bullet}$ for any $k\geq 1$;
\item \label{L:25:phi ss odd:upper}
$\varphi^{(k)}(s_is_j^{-1}) = (s_is_j^{-1})^{\bullet}$ for any
$2\leq k\leq 2n$.
\end{enumerate}
\end{enumerate}
\end{lemma}
\begin{proof}
Let $z = (y_kx_k)^{n}$. Then
$$
gT_k = g\cdot [y_k^{-1},z] = gy_k^{-1}\cdot y_k^{z^{-1}}
$$
which proves \ref{L:25:gT}. 
\begin{align*}
gT_khT_k 
&= g\cdot[y_k^{-1},z]\cdot h\cdot[x_k,z]\\
&= gy_k^{-1}\cdot y_k^{z^{-1}}\cdot hx_k\cdot (x_k^{-1})^{z^{-1}}\\
&= gy_k^{-1}\cdot (hx_k)^{\bullet}\cdot (y_kx_k^{-1})^{z^{-1}}
\end{align*}
which proves \ref{L:25:gThT}. 
If $k$ is even then 
$\varphi_k(s_i^{-1}s_j) = T_k^{-1}s_i^{-1}s_jT_k=(s_i^{-1}s_j)^{\bullet}$ which proves \ref{L:25:phi ss even}.  
Finally, assume that $i$ and $j$ have the same parity. 
If $k$ is even or if $i$ and $j$ are even then
$$
\varphi_k(s_is_j^{-1}) = s_iT_kT_k^{-1}s_j^{-1}=s_is_j^{-1}.
$$
If $i,j,k$ are all odd then
$$
\varphi_k(s_is_j^{-1}) = T_k^{-1}s_is_j^{-1}T_k = (s_is_j^{-1})^{\bullet}
$$
which proves \ref{L:25:phi ss odd}\ref{L:25:phi ss odd:lower}. 
We prove \ref{L:25:phi ss odd}\ref{L:25:phi ss odd:upper} by downward induction on $2 \leq k \leq 2n$. Since 
$\varphi^{(2n)}={\rm Id}$ the base of the induction is true and the induction
step follows from \ref{L:25:phi ss odd}\ref{L:25:phi ss odd:lower}:
\[
\varphi^{(k)}(s_is_j^{-1})
=\varphi^{(k+1)}\varphi_{k+1}(s_is_j^{-1}) \\
=\varphi^{(k+1)}(s_is_j^{-1})^{\bullet}\\
=(s_is_j^{-1})^{\bullet}.
\]
\end{proof}

\paragraph{The linear upper bound in \eqref{Eq:norm-gn2}.}
\begin{proposition}\label{P:fiuu}
Let $u_0\in S$ be a generator with even index and let $u_1\in S$ be a generator with odd index. 
Then
$$
\|\varphi^{(2n-2k)}(u_0u_1^{\pm 1})\| \leq 4k+2,
$$
for any $0\leq k \leq n-1$.
In particular, for $k=n-1$ we get that for $g_n=\varphi^{(2)}(x_0)$
\[
\|g_n^2\|= \| \varphi^{(2)}(x_0^2)\| \leq \| \varphi^{(2)}(x_0 x_1)\| + \| \varphi^{(2)}(x_1^{-1}x_0) \| = \| \varphi^{(2)}(x_0 x_1)\| + \| \varphi^{(2)}(x_0x_1^{-1}) \|  \leq 8n-4
\]
which proves the upper bound in \eqref{Eq:norm-gn2}.
\end{proposition}
\begin{proof}
We will prove the following two inequalities
\begin{align}
\|\varphi^{(2n-\ell)}(u_0u_1)\| &\leq 4 \left\lfloor \frac{\ell+1}{2}\right\rfloor + 2\label{Eq:[k+1]}\\
\|\varphi^{(2n-\ell)}(u_0u_1^{-1})\| &\leq 4 \left\lfloor \frac{\ell}{2}\right\rfloor + 2\label{Eq:[k]}.
\end{align}
for every $0\leq \ell\leq 2n-2$. Both are proven by induction on $\ell$. For $\ell=0$ we have
$\varphi^{(2n)}={\rm Id}$ and hence the base of the induction is true in both cases.
Assume both inequalities are true for some $0\leq \ell < 2n-2$.
We prove the induction step for \eqref{Eq:[k+1]}.
First,
\begin{align*}
\varphi^{(2n-\ell-1)}(u_0u_1) 
&= \varphi^{(2n-\ell)}\varphi_{2n-\ell}(u_0u_1).
\end{align*}
If $\ell$ is odd then 
$\varphi_{2n-\ell}(u_0u_1) = u_0T_{2n-\ell}T_{2n-\ell}^{-1}u_1=u_0u_1$ 
by definition and therefore by induction hypothesis,
\begin{equation}
\|\varphi^{(2n-\ell-1)}(u_0u_1)\| = \|\varphi^{(2n-\ell)}(u_0u_1)\| 
\leq 4 \left\lfloor \frac{\ell+1}{2}\right\rfloor +2 
=4 \left\lfloor \frac{\ell+2}{2}\right\rfloor+2.
\label{Eq:lodd}
\end{equation}
If $\ell$ is even then by Lemma \ref{L:25}\ref{L:25:gThT}
$$
\varphi_{2n-\ell}(u_0u_1)
=u_0T_{2n-\ell}u_1T_{2n-\ell} 
=u_0y_{2n-\ell}^{-1}\cdot (u_1x_{2n-\ell})^{\bullet}
\cdot (y_{2n-\ell}x_{2n-\ell}^{-1})^{\bullet}.
$$
For each of the factors on the right hand side we have the following estimates,
\begin{align*}
\|\varphi^{(2n-\ell)}(u_0y_{2n-\ell}^{-1})\| &\leq 2 
&\text{by Lemma \ref{L:25}\ref{L:25:phi ss odd}\ref{L:25:phi ss odd:upper}}\\
\|\varphi^{(2n-\ell)}(u_1x_{2n-\ell})^{\bullet}\| 
&\leq 4 \left\lfloor \frac{\ell+1}{2}\right\rfloor +2 &\text{by induction hypothesis}\\
\|\varphi^{(2n-\ell)}(y_{2n-\ell}x_{2n-\ell}^{-1})^{\bullet}\| 
&\leq 2 &\text{by Lemma \ref{L:25}\ref{L:25:phi ss odd}\ref{L:25:phi ss odd:upper}}.
\end{align*}
Consequently, we get
$$
\|\varphi^{(2n-\ell-1)}(u_0u_1)\| \leq 4 \left\lfloor \frac{\ell+1}{2}\right\rfloor +6
= 4 \left\lfloor \frac{\ell+2}{2}\right\rfloor + 2
$$
which, together with \eqref{Eq:lodd} 
proves the induction step for $\ell$ both odd and even, 
and hence proves~\eqref{Eq:[k+1]}.

The proof of the induction step for \eqref{Eq:[k]} is similar and we present
the main inequalities without going into the details.
\begin{align*}
\|\varphi^{(2n-\ell-1)}(u_0u_1^{-1})\|
&= \|\varphi^{(2n-\ell)}\varphi_{2n-\ell}(u_0u_1^{-1})\|\\
&=
\begin{cases}
\|\varphi^{(2n-\ell)}(u_0u_1^{-1}) \| &\text{if $\ell$ is even}\\
\|\varphi^{(2n-\ell)}(u_0y_{2n-\ell}{}^{-1}\cdot (u_1^{-1}x_{2n-\ell})^{\bullet}
\cdot (y_{2n-\ell}x_{2n-\ell}{}^{-1})^{\bullet}) \| 
&\text{if $\ell$ is odd}
\end{cases}
\\
& \leq 
\begin{cases}
4\left\lfloor\frac{\ell}{2}\right\rfloor +2  &\text{if $\ell$ is even}\\
4\left\lfloor \frac{\ell}{2}\right\rfloor + 6 
&\text{if $\ell$ is odd}
\end{cases}
\\
&= 4\left\lfloor \frac{\ell+1}{2}\right\rfloor + 2. 
\end{align*}
For $\ell = 2k$ both inequalities \eqref{Eq:[k+1]} and \eqref{Eq:[k]} 
yield
$$
\|\varphi^{(2n-2k)}(u_0u_1^{\pm 1})\|\leq 4k + 2
$$
which finishes the proof.
\end{proof}

\paragraph{The quadratic upper bound in \eqref{Eq:norm-gn}.}

\begin{proposition}\label{P:fiu}
Let $0\leq k\leq n-1$. If $u\in S$ is a generator with even index
then
\begin{equation}
\|\varphi^{(2n-2k)}(u)\| \leq 4k^2+1.
\label{Eq:fiu}
\end{equation}
In particular, for $k=n-1$ and $g_n=\varphi^{(2)}(x_0)$ defined in \eqref{Eq:gn}, we get that
\[
\|g_n\|=\|\varphi^{(2)}(x_0)\|\leq 4n^2-8n+5
\]
which  proves the upper bound of \eqref{Eq:norm-gn}.
\end{proposition}
\begin{proof}
The proof is by induction on $k$. The base, for $k=0$, is trivially true
since $\varphi^{(2n)}={\rm Id}$. Assume the statement is true for $0\leq k <n-1$.
We have 
$$
\varphi^{(2n-2k-2)} = \varphi^{(2n-2k)}\circ \varphi_{2n-2k}\circ \varphi_{2n-2k-1}
$$
and we will analyse the values $(\varphi_{2n-2k}\circ \varphi_{2n-2k-1})(u)$.
Set $i=2n-2k-1$. 
Since the generator $u$ is of even index, by Lemma \ref{L:25}\ref{L:25:gT}
$$
\varphi_{2n-2k-1}(u) = uT_i = uy_i^{-1}\cdot y_i^{\bullet}.
$$
Furthermore, since $\varphi_{2n-2k}$ is a homomorphism
\begin{align*}
\varphi_{2n-2k}(uy_i^{-1}) &= uy_i^{-1} & & \text{by definition of $\varphi_{2n-2k}$} \\
\varphi_{2n-2k}(y_i^{\bullet}) &= \varphi_{2n-2k}(y_i)^{\bullet} = (y_iy_{i+1}^{-1})^{\bullet} \cdot (y_{i+1})^{\bullet} & & \text{by  Lemma \ref{L:25}\ref{L:25:gT}.}
\end{align*}
We have the following estimates
\begin{align*}
\|\varphi^{(2n-2k)}(uy_i^{-1})\| &\leq 4k+2 &&\text{by Proposition \ref{P:fiuu}}\\
\|\varphi^{(2n-2k)}(y_iy_{i+1}^{-1})^{\bullet}\| &\leq 4k+2 &&\text{by Proposition \ref{P:fiuu}}\\
\|\varphi^{(2n-2k)}(y_{i+1})^{\bullet}\| &\leq 4k^2+1 &&\text{by induction hypothesis}.
\end{align*}
It follows that
\begin{align*}
\|\varphi^{(2n-2k-2)}(u)\|
&\leq 4k+2 + 4k+2 + 4k^2+1 = 4(k+1)^2+1.
\end{align*}
This finishes the induction step and the proof of the main inequality.
\end{proof}

\subsection{The lower bounds}

\paragraph{The homomorphisms $\pi$ and $\rho$.}
The proofs of the lower bounds are more involved and require deeper understanding
of the homomorphisms $\varphi^{(k)}\colon \B F_{\infty}\to \B F_{\infty}$.
We start with introducing some auxiliary homomorphisms and some notation. 
Let $I\subseteq \mathbb N$ be 
a subset of non-negative integers. Let $S_I = \{x_i,y_i\in S\ |\ i\in I\}$
and 
\[ 
\B F_I = \B F(S_I)\leq \B F_{\infty}
\]
 be the free subgroup generated by~$S_I$.
For example, $T_k \in \B F_{\{ k\}}$.
Let $\pi_I\colon \B F_{\infty}\to \B F_I$ and 
$\widehat{\pi}_I\colon \B F_{\infty}\to \B F_{\mathbb N\setminus I}$
be surjective homomorphisms defined by
$$
\pi_I(s_i)=
\begin{cases}
s_i & \text{if $i\in I$}\\
1 & \text{if $i\notin I$}
\end{cases}
\qquad\text{and}\qquad
\widehat{\pi}_I(s_i)=
\begin{cases}
1 & \text{if $i\in I$}\\
s_i & \text{if $i\notin I$}.
\end{cases}
$$
For any $k\in \mathbb N\cup \{\infty\}$ let 
$\rho_k\colon \B F_{\infty}\to \B F_{\infty}$ be defined by
$$
\rho_k(s_i) =
\begin{cases}
x_{i\,{\rm mod}\,2} &\text{ if }i\leq k\\
s_i &\text{ if } i>k
\end{cases}
$$
Note that $\rho_{\infty}(s_i) = x_{i\,{\rm mod}\,2}$ for all $i \geq 0$.
In particular,
the image of $\rho_{\infty}$ is equal to the free subgroup of rank $2$ generated by $\{x_0,x_1\}$.
It follows directly from the definition that if $k>j$ then 
\begin{equation}\label{eqn:phi_k(T_j)}
\rho_k(T_j) =
\rho_k((x_jy_j)^n (y_jx_j)^{-n}) =
(x_{j \, {\rm mod} \, 2})^{2n} (x_{j \, {\rm mod} \, 2})^{-2n} =1.
\end{equation}

For any $a,b \in \mathbb N$ let $[a,b]$ denote the interval $\{a, a+1, \dots, b\}$ in $\mathbb{N}$.
We write $[a,b]_{\rm ev}$ and $[a,b]_{\rm odd}$ for the subset of even and odd numbers in $[a,b]$.
Given a finite $I \subseteq \mathbb N$ of the form $I=\{i_1< \dots < i_k\}$ and a collection $g_i \in \B F_\infty$ for each $i \in I$, we denote
\[
\prod_{i \in I^{\rm op}} g_i = g_{i_k} \cdots g_{i_2} \cdot g_{i_1},
\]
the product of the $g_{i_j}$'s in the ``reverse'' order.
For any integer $k$ denote by $k_{\rm ev}$ the smallest even integer $\geq k$ and by $k_{\rm odd}$ the smallest odd integer $\geq k$.

\begin{lemma}\label{L:fi(k)}
Let $2\leq k\leq 2n$. 
The value of the homomorphism $\varphi^{(k)}\colon \B F_\infty\to \B F_\infty$ on a generator $s_i\in S$ is given by
\begin{equation}
\varphi^{(k)}(s_i)=
\begin{cases}
s_i\cdot U_{k} & \text{if $i$ is even}\\
V_{k}^{-1} \cdot s_i\cdot W_{k} & \text{if $i$ is odd},
\end{cases}
\label{Eq:fi(k)}
\end{equation}
where
\begin{align*}
U_{k} 
&= 
\prod_{j \in [k+1,2n]^{\rm op}} \varphi^{(j)}(T_j) = 
\varphi^{(2n)}(T_{2n}) \cdot \varphi^{(2n-1)}(T_{2n-1}) \dots  \varphi^{(k+1)}(T_{k+1})
\\
V_{k}
&= 
\prod_{j \in [k+1,2n]_{\rm odd}^{\rm op}} \varphi^{(j)}(T_j) = 
\varphi^{(2n-1)}(T_{2n-1}) \cdot \varphi^{(2n-3)}(T_{2n-3}) \dots  \varphi^{((k+1)_{\rm odd})}(T_{(k+1)_{\rm odd}})
\\
W_{k} 
&= 
\prod_{j \in [k+1,2n]_{\rm ev}^{\rm op}} \varphi^{(j)}(T_j) = 
\varphi^{(2n)}(T_{2n}) \cdot \varphi^{(2n-2)}(T_{2n-2}) \dots  \varphi^{((k+1)_{\rm ev})}(T_{(k+1)_{\rm ev}})
\end{align*}
Moreover, $U_k,V_k,W_k\in \B F_{\{k+1,\ldots,2n\}}=\langle x_{k+1},y_{k+1},\ldots,x_{2n},y_{2n} \rangle \leq \B F_\infty$  and for every $k\in \mathbb N$ 
\[
\rho_{\infty}(U_k)=\rho_{\infty}(V_k)=\rho_{\infty}(W_k)=1.
\]

\end{lemma}
\begin{proof}
We use downward induction on $2 \leq k \leq 2n$.
The base $k=2n$ clearly holds since $\varphi^{(2n)}(s_i)=s_i$ (it is the identity map) and since $U_{2n}=V_{2n}=W_{2n}=1$.

Assume inductively the result holds for $2<k \leq 2n$ and we prove it for $k-1$.
Assume first that $i$ is even.
Then by the induction hypothesis
\[
\varphi^{(k-1)}(s_i) 
= 
\varphi^{(k)}(\varphi_k(s_i)) 
=
\varphi^{(k)}(s_iT_k)
=
\varphi^{(k)}(s_i)  \varphi^{(k)}(T_k)
=
s_i U_k \cdot \varphi^{(k)}(T_k)
=
s_i U_{k-1},
\]
which proves the induction step in this case.
The proof of the induction step if $i$ is odd is similar, distinguishing the cases that $k$ is even or odd.
The details are straightforward and left to the reader.

Since $\varphi_\ell(w)$ is obtained from a word $w$ by inserting the symbols $x_\ell^{\pm 1}$ and $y_\ell^{\pm 1}$ it is clear that for any $k+1 \leq j \leq 2n$
\[
\varphi^{(j)}(T_j) = \varphi_{2n} \dots \varphi_{j+1}(T_j) \in \B F_{\{ k+1, \dots,2n\}}
\]
and therefore $U_k, V_k, W_k \in \B F_{\{ k+1, \dots,2n\}}$.
For the final assertion of the lemma, observe that \eqref{eqn:phi_k(T_j)} implies that for any $\ell\geq 1$ and any $s_i \in S$
\[
\rho_\infty (\varphi_\ell(s_i))
=
\rho_{\infty}
\left(
\left\{
\begin{array}{ll}
s_iT_\ell & \text{if either $i$ or $\ell$ is even}
\\
T_\ell^{-1}s_i & \text{if both $i$ and $\ell$ are odd}
\end{array}
\right.
\right)
=\rho_\infty(s_i).
\]
It follows that $\rho_\infty \circ \varphi_\ell =\rho_\infty$ for any $\ell\geq 1$.
Then $\rho_\infty \varphi^{(j)} = \rho_\infty \varphi_{2n} \dots \varphi_{j+1}=\rho_\infty$, and it follows from \eqref{eqn:phi_k(T_j)} that for any $2 \leq j \leq 2n$
\[
\rho_\infty (\varphi^{(j)}(T_j)) = \rho_\infty(T_j)=1.
\]
Therefore $\rho_\infty(U_k), \rho_\infty(V_k), \rho_\infty(W_k)=1$ as required.
\end{proof}

The following lemma shows various relations between the homomorphisms $\varphi^{(k)}$
and the ones defined above.

\begin{lemma}\label{L:fi-pi}
Let $2\leq k\leq 2n$.
\begin{enumerate}[label=(\arabic*)]
\item \label{L:fi-pi:1}
$\pi_I\circ \varphi^{(k)}=\pi_I$ for any $I\subseteq\{0,1,\ldots,k\}$.
\item \label{L:fi-pi:2}
$\widehat{\pi}_{\{j\}}\circ\varphi^{(k)}\circ\varphi_j(g)=\varphi^{(k)}(g)$ for
any $1\leq j\leq k$ and any $g \in \B F_{\mathbb N\setminus\{j\}}$.
\item \label{L:fi-pi:3}
$\rho_k\circ\varphi^{(k)} = \varphi^{(k)}\circ \rho_k$.
\end{enumerate}
\end{lemma}
\begin{proof}
\ref{L:fi-pi:1} 
It suffices to prove that $\pi_I(\varphi^{(k)}(s))= \pi_I(s)$ for all $s \in S$.
It follows from Lemma \ref{L:fi(k)} that $\varphi^{(k)}(s) = vsw$ for some $v,w\in \B F_{\{k+1,\ldots,2n\}}$. 
Since
$I\subseteq\{0,\ldots,k\}$ it is clear that $\pi_I(v),\pi_I(w)=1$.
Hence
$$
\pi_I(\varphi^{(k)}(s))=\pi_I(vsw) = \pi_I(s).
$$

\noindent
\ref{L:fi-pi:2}
It suffices to prove the equality in the special case $g=s_i \in S$ where $i \neq~j$.
By construction $\varphi_j(s_i) = s_iT_j$ or $T_j^{-1}s_i$ depending on the parity of $i$ and $j$. 
Then $\varphi^{(k)}(\varphi_j(s_i))=\varphi^{(k)}(s_i) \cdot \varphi^{(k)}(T_j)$ or $\varphi^{(k)}(\varphi_j(s_i))= \varphi^{(k)}(T_j)^{-1} \cdot \varphi^{(k)}(s_i)$.
It remains to show that $\widehat{\pi}_{\{j\}}(\varphi^{(k)}(T_j))=1$.

Clearly $T_j=(x_jy_j)^n (y_jx_j)^{-n}$ is in the commutator subgroup of $\langle x_j,y_j \rangle$, so it is a product of conjugates of $[x_j^{\pm 1},y_j^{\pm 1}]$.
So it remains to show that $\widehat{\pi}_{\{j\}}(\varphi^{(k)}([x_j^{\pm 1},y_j^{\pm 1}]))=1$.
By Lemma \ref{L:fi(k)}, $\varphi^{(k)}(s_j) = vs_jw$ for some $v,w\in \B F_{\{k+1,\ldots,2n\}}$ which only depend on $k$ and $j \mod 2$. 
Then 
\[
\widehat{\pi}_{\{j\}}(\varphi^{(k)}([x_j^{\pm 1},y_j^{\pm 1}]))
=
\widehat{\pi}_{\{j\}} ([(vx_jw)^{\pm 1}, (vy_jw)^{\pm 1}]) 
=
[ \widehat{\pi}_{\{j\}}(vw)^{\pm 1} , \widehat{\pi}_{\{j\}}(vw)^{\pm 1} ] =1.
\]
\ref{L:fi-pi:3} 
Notice that $\rho_k(g) = g$ for any $g\in \B F_{\{k+1,\ldots,2n\}}$, by definition of $\rho_k$. 
By Lemma \ref{L:fi(k)}, for any $s_i \in S$ there exist $v,w \in \B F_{\{k+1,\dots,2n\}}$, which only depend on $k$ and $i \mod 2$, such that $\varphi^{(k)}(s_i)=vs_iw$.
Hence, $\rho_k(v)=v$ and $\rho_k(w)=w$.
Since $\rho_k(s_i)=x_{i \, {\rm mod} \, 2}$ if $i \leq k$ and $\rho_k(s_i)=s_i$ if $i>k$, the same lemma implies that $\varphi^{(k)}(\rho_k(s_i))=v \rho_k(s_i)w$.
We deduce that 
\[
\rho_k(\varphi^{(k)}(s_i))
=\rho_k(vs_iw)
=v\rho_k(s_i)w
=\varphi^{(k)}(\rho_k(s_i)).
\]
The result follows.
\end{proof}

\paragraph{The homomorphisms $\psi_k$.}
Let $H=\langle x_0,x_1\rangle\leq \B F_{\infty}$. 
For any {\em odd} integer $k \geq 3$ define a homomorphism $\psi_k\colon H\to \B F_{\infty}$ by
$$
\psi_k = \pi_{\{k,k+1\}}\circ \varphi_{k+1}\circ\varphi_{k}.
$$

Consider a free group on a set $T$ of generators.
If $g$ is a reduced word we can write it in the form $g=uvu^{-1}$ where $u,v$ are subwords of $g$ and $u$ has maximal length with respect to such presentation.
It is then clear that $g^n =uv^nu^{-1}$ and that if $n \neq 0$ then this is a reduced word which involves {\em exactly} the same letters in $T$ as $g$.
We also recall that the centraliser of any non-trivial element $g$ in a free group $\B F(T)$ is a cyclic group.
Since $1 \neq g \in C_{\B F(T)}(g)$, it follows from the argument above that every non-trivial element of $C_{\B F(T)}(g)$ involves exactly the same generators as $g$.

\begin{lemma}\label{L:psik}
The homomorphism $\psi_k$ is injective.
Its image is contained in the commutator subgroup
$[\B F_{\{k,k+1\}},\B F_{\{k,k+1\}}]$. In particular, every
element in the image of $\psi_k$ has even conjugation-invariant norm.
\end{lemma}
\begin{proof}
Since $k$ is odd and bigger than $1$ we have the following computation.
\begin{align*}
\psi_k(x_0) 
&=\pi_{\{k,k+1\}}(\varphi_{k+1}(\varphi_k(x_0)))
=\pi_{\{k,k+1\}}(\varphi_{k+1}(x_0\cdot T_k))\\
&=\pi_{\{k,k+1\}}(x_0\cdot T_{k+1}\varphi_{k+1}(T_k))
=T_{k+1}\cdot \varphi_{k+1}(T_k)\\
\psi_k(x_1) 
&=\pi_{\{k,k+1\}}(\varphi_{k+1}(\varphi_k(x_1)))
=\pi_{\{k,k+1\}}(\varphi_{k+1}(T_k^{-1}\cdot x_1))\\
&=\pi_{\{k,k+1\}}(\varphi_{k+1}(T_k^{-1})\cdot x_1\cdot T_{k+1})
=\varphi_{k+1}(T_k)^{-1}\cdot T_{k+1}.
\end{align*}
This shows that the image of $\psi_k$ is contained in the subgroup
generated by the elements
$T_{k+1},\varphi_{k+1}(T_k)\in \B F_{\{k,k+1\}}$. 
Since both of them are commutators, the last claims follow.

To show injectivity of $\psi_k$, recall that free groups of finite rank are Hopfian, see \cite[Theorem 2.13]{zbMATH02146477}. 
That is, any {\em proper} quotient of a free group of finite rank cannot be isomorphic to a free group of the same rank.
It follows that any free  proper quotient of $H\cong \B F(2)$ is either infinite cyclic or trivial. 
Hence, if $\psi_k$ was not injective then its image would be
abelian and we would have
$$
1=[\psi_k(x_0),\psi_k(x_1)] = \left[T_{k+1}^2,\varphi_{k+1}(T_k)^{-1}\right].
$$
Since $T_{k+1}^2\in \langle x_{k+1},y_{k+1}\rangle = \B F_{\{k+1\}}$, this implies that $\varphi_{k+1}(T_k) \in C_{\B F_\infty}(T_{k+1}^2) \subseteq \B F_{\{k+1\}}$.
This is a contradiction because $\pi_{\{k\}}(\varphi_{k+1}(s_k))=\pi_{\{k\}}(s_k T_{k+1})=s_k$ so $\pi_{\{k\}}(\varphi_{k+1}(T_k))=T_k \neq 1$ since $T_k=(x_ky_k)^n(y_kx_k)^{-n}$.
\end{proof}

\paragraph{The linear lower bound in \eqref{Eq:norm-gn2}.}

\begin{proposition}\label{P:fiw-linear}
Let $0\leq k\leq n-1$ and let $h\in H=\langle x_0,x_1\rangle\leq \B F_\infty$ 
be a non-trivial element of even length.
Then
$$
\|\varphi^{(2n-2k)}(h)\|\geq 2k+2.
$$
In particular, $\|g_n^2\|=\|\varphi^{(2)}(x_0^2)\|\geq 2n$, which is the lower bound in \eqref{Eq:norm-gn2}.
\end{proposition}
\begin{remark}
A similar linear bound holds for an $h$ of odd length, but we will
prove stronger quadratic estimate in Proposition \ref{P:final} below.
\end{remark}

\begin{proof}[Proof of Proposition \ref{P:fiw-linear}]
The proof is by induction on $k$. For $k=0$ the statement is
trivially true since $\varphi^{(2n)}={\rm Id}$ and $\|h\|\geq 2$, since
it is non-trivial and represented by a word of even length. Assume 
the result is true for $0\leq k\leq n-2$ and prove it for $k+1$.

Let $j=2n-2k-1$. 
Then 
$\varphi^{(2n-2k-2)} = \varphi^{(j+1)}\circ\varphi_{j+1}\circ\varphi_j$.
Let $1\neq h\in H$ and let $w$ be a word representing $\varphi^{(2n-2k-2)}(h)$.
Let $c$ be a minimal cancellation sequence for $w$. Then
$\pi_{\{j,j+1\}}(c)$ and $\widehat{\pi}_{\{j,j+1\}}(c)$ are cancellation
sequences for 
$\pi_{\{j,j+1\}}(w)$ and $\widehat{\pi}_{\{j,j+1\}}(w)$, respectively.
It follows that
\begin{equation}
\|w\| = \ell(c) = 
\ell(\pi_{\{j,j+1\}}(c))+ \ell(\widehat{\pi}_{\{j,j+1\}}(c))
\geq \|\pi_{\{j,j+1\}}(w)\|+ \|\widehat{\pi}_{\{j,j+1\}}(w)\|.
\label{Eq:norm-w}
\end{equation}
Furthermore,
\begin{align*}
\pi_{\{j,j+1\}}(w) 
&= \pi_{\{j,j+1\}}\varphi^{(j+1)}\varphi_{j+1}\varphi_j(h)\\
&= \pi_{\{j,j+1\}}\varphi_{j+1}\varphi_j(h) &\text{by Lemma \ref{L:fi-pi}\ref{L:fi-pi:1}}\\
&=\psi_j(h)
\end{align*}
and it follows from Lemma \ref{L:psik} that $\psi_j(h)$ is non-trivial
and its norm is even.
In particular, 
\begin{equation}
\|\pi_{\{j,j+1\}}(w)\|\geq 2.
\label{Eq:w>2}
\end{equation}
In the following computation we apply Lemma \ref{L:fi-pi}\ref{L:fi-pi:2}
twice, noting that $j \geq 1$.
\begin{align}
\widehat{\pi}_{\{j,j+1\}}(w) 
&=\widehat{\pi}_{\{j\}}\widehat{\pi}_{\{j+1\}}
\varphi^{(j+1)}\varphi_{j+1}\varphi_j(h)\nonumber\\
&=\widehat{\pi}_{\{j\}}\varphi^{(j+1)}\varphi_j(h)\nonumber \\
&=\varphi^{(j+1)}(h)=\varphi^{(2n-2k)}(h).
\label{Eq:piwfih}
\end{align}
Finally, we get the following computation which proves the induction step and finishes
the proof.
\begin{align*}
\|\varphi^{(2n-2k-2)}(h)\|
&=\|w\| 
\geq \|\pi_{\{j,j+1\}}(w)\|+ \|\widehat{\pi}_{\{j,j+1\}}(w)\| &\text{by \eqref{Eq:norm-w}}\\
&\geq 2 + \|\varphi^{(2n-2k)}(h)\| &\text{by \eqref{Eq:w>2} and \eqref{Eq:piwfih}}\\
&\geq 
2(k+1)+2 
&\text{by induction hypothesis}
\end{align*}
\end{proof}

\paragraph{Preparation for the proof of the quadratic lower bound in \eqref{Eq:norm-gn}.}

Let $2\leq k \leq 2n$ and let $w=t_1\dots t_m$ be a word in generators 
from $S_{\{0,\ldots,k\}}^{\pm 1}$.
Since $\varphi^{(k)}$ inserts elements $T_{\ell}$ with $\ell>k$, which are words in
$S_{\{k+1,\ldots ,2n\}}$, the element $\varphi^{(k)}(w)$ contains a unique copy
of $w$ as a subsequence.
Let $(c,\C F)$ be a cancellation system for $\varphi^{(k)}(w)$ with minimal $c$.
Suppose that the folding $\C F$ contains $\{t_i,t_j\}$, where $i<j$, from the subsequence 
$w$ of $\varphi^{(k)}(w)$
and let 
\begin{align*}
w_L &= t_1 \cdots t_{i-1}\\
w_0 &= t_i\cdots t_{j}\\
w_R &= t_{j+1} \cdots t_{m}.
\end{align*}
Recall the homomorphisms $\rho_k$.
We note that $\rho_{\infty}\colon \B F_{\infty}\to \langle x_0,x_1\rangle\leq \B F_{\infty}$ is a homomorphism defined by $\rho_{\infty}(s_i)=x_{i\,{\rm mod}\,2}$.

\begin{lemma}\label{L:fik-anti-triangle}
With the notation and assumptions above we have
$$
\|\varphi^{(k)}(w)\|\geq \|\varphi^{(k)}(\rho_{\infty}(w_0))\| + \|\varphi^{(k)}(\rho_{\infty}(w_Lw_R))\|.
$$
\end{lemma}
\begin{proof}
Recall from Lemma \ref{L:fi(k)} that $\varphi^{(k)}(t_i)=gt_ih$, for some $g,h\in \B F_{\{k+1,\ldots,2n\}}$. 
Since $t_j=t_i^{-1}$, we have
$\varphi^{(k)}(t_j) = h^{-1}t_jg^{-1}$. Let $v_L,v_0$ and $v_R$ be the subwords of $\varphi^{(k)}(w)$ defined by
$$
\varphi^{(k)}(w) =
\underbrace{\varphi^{(k)}(t_1\cdots t_{i-1}) \cdot g}_{v_L}
\underbrace{t_ih \cdot \varphi^{(k)}(t_{i+1}\cdots t_{j-1}) \cdot h^{-1}t_j}_{v_0}
\underbrace{g^{-1}\cdot \varphi^{(k)}(t_{j+1}\cdots t_m)}_{v_R}.
$$
By Lemma \ref{L:folding}, 
$
\|\varphi^{(k)}(w)\| \geq \|v_0\| + \|v_L v_R\|$.
Also, we have the following computations
\begin{align*}
\rho_k(gv_0g^{-1}) 
&=\rho_k\left(gt_ih\varphi^{(k)}(t_{i+1}\cdots t_{j-1})h^{-1}t_jg^{-1}\right)\\
&=\rho_k\varphi^{(k)}(t_it_{i+1}\cdots t_{j-1}t_j)\\
&=\varphi^{(k)}\rho_k(t_i\cdots t_j) &\text{by Lemma \ref{L:fi-pi}\ref{L:fi-pi:3}}\\
&=\varphi^{(k)}(\rho_{\infty}(w_0)).
\\
\rho_k(v_Lv_R)
&=\rho_k(\varphi^{(k)}(t_1\cdots t_{i-1})g\cdot g^{-1}\varphi^{(k)}(t_{j+1}\cdots t_m))\\
&=\varphi^{(k)}\rho_k(t_1\cdots t_{i-1} \cdot t_{j+1}\cdots t_m)  &\text{by Lemma \ref{L:fi-pi}\ref{L:fi-pi:3}} \\
&=\varphi^{(k)}(\rho_{\infty}(w_Lw_R)).
\end{align*}
Since $\rho_k \colon \B F_\infty \to \B F_\infty$ carries a generator to a generator, it is Lipschitz with constant $1$.
Therefore,
\begin{align*}
\| \varphi^{(k)}(w) \| & \geq \| v_0 \| + \| v_L v_R\|
\\
& = 
\| gv_0g^{-1} \| + \| v_L v_R\|
\\
& \geq 
\| \rho_k(gv_0g^{-1}) \| + \| \rho_k(v_L v_R)\|
\\
& =
\| \varphi^{(k)}(\rho_{\infty}(w_0)) \| + \| \varphi^{(k)}(\rho_{\infty}(w_Lw_R))\|.
\end{align*}
This completes the proof.
\end{proof}

\begin{lemma}\label{L:intricate}
Let $3\leq k\leq 2n-1$ be an odd integer. Let $w\in H=\langle x_0,x_1\rangle$ be a non-trivial
reduced word which starts and ends with the same letter, including sign.
Then either
$$
\|\varphi^{(k-1)}(w)\|\geq 4n + \|\varphi^{(k+1)}(w)\|
$$
or there exist $w_0,w_L,w_R\in H$ such that 
\begin{enumerate}[label=(\alph*)]
\item $w=w_Lw_0w_R$; equality in $H$, not necessarily a decomposition of $w$ into the concatenation of three subwords. 

\item $w_0\neq 1$ and $w_Lw_R\neq 1$, and 
$$
\|\varphi^{(k-1)}(w)\|\geq  \|\varphi^{(k+1)}(w_0)\|+ \|\varphi^{(k+1)}(w_Lw_R)\|.
$$
\end{enumerate}
\end{lemma}
\begin{proof}
Let $w=t_1\dots t_m$, where $t_i\in \{x_0^{\pm 1},x_1^{\pm 1}\}$. 
After replacing $w$ with
$w^{-1}$, we can assume, without loss of generality, that $w$ starts and ends 
with either $x_0$ or $x_1$. The proof has
two parts depending on whether $t_1$ is equal to $x_0$ or $x_1$. 
We will conduct both cases in parallel as follows.

Since the homomorphisms
$\widehat{\pi}_{\{k\}}$ and $\widehat{\pi}_{\{k+1\}}$ 
map generators to generators, they are Lipschitz with
constant $1$, which is the first inequality in the following
parallel computation; The second equalities follow from Lemma \ref{L:fi-pi}\ref{L:fi-pi:2}
$$
\begin{array}{cc}
\parbox[t]{0.45 \textwidth}{
\begin{align*}
t_1 &=x_0\\
\hline\\
\|\varphi^{(k-1)}(w)\|
&\geq \|\widehat{\pi}_{\{k+1\}}\varphi^{(k-1)}(w)\|\\
&=\|\widehat{\pi}_{\{k+1\}}\varphi^{(k+1)}\varphi_{k+1}\varphi_k(w)\|\\
&=\|\varphi^{(k+1)}\varphi_k(w)\|.
\end{align*}
}&
\parbox[t]{0.45 \textwidth}{
\begin{align*}
t_1 &= x_1\\
\hline\\
\|\varphi^{(k-1)}(w)\|
&\geq \|\widehat{\pi}_{\{k\}}\varphi^{(k-1)}(w)\|\\
&=\|\widehat{\pi}_{\{k\}}\varphi^{(k)}\varphi_k(w)\|\\
&=\|\varphi^{(k)}(w)\|\\
&=\|\varphi^{(k+1)}\varphi_{k+1}(w)\|.
\end{align*}
}
\end{array}
$$
Let $j = k$ if $t_1=x_0$ and $j = k+1$ if $t_1=x_1$.
Recall that $k$ is odd, so 
\[
t_1 = t_m=x_{(j+1)\,{\rm mod}\,2}.
\]
In this notation both estimates above can be written together as
\begin{equation}
\|\varphi^{(k-1)}(w)\|\geq \|\varphi^{(k+1)}\varphi_{j}(w)\|
\label{Eq:j-notation}
\end{equation}
and we will now analyse the element $\varphi^{(k+1)}\varphi_{j}(w)$.

Let $(c,\C F)$ be a cancellation system for
$\varphi^{(k+1)}\varphi_{j}(w)$ 
with minimal $c$.
Then $\pi_{\{j\}}(c)$ and $\widehat{\pi}_{\{j\}}(c)$ are cancellation 
sequences for
$\pi_{\{j\}}(\varphi^{(k+1)}\varphi_{j}(w))$ and 
$\widehat{\pi}_{\{j\}}(\varphi^{(k+1)}\varphi_{j}(w))$, respectively.
Moreover, since $j=k$ or $j=k+1$, by Lemma \ref{L:fi-pi}\ref{L:fi-pi:2}
$$
\widehat{\pi}_{\{j\}}(\varphi^{(k+1)}\varphi_{j}(w)) 
=\varphi^{(k+1)}(w).
$$
It follows that
\[
\ell(\widehat{\pi}_{\{j\}}(c))\geq \|\varphi^{(k+1)}(w)\|.
\]
\subsubsection*{Case 1: $\ell(\pi_{\{j\}}(c))\geq 4n$.}
Then
\begin{align*}
\|\varphi^{(k-1)}(w)\|
&\geq\|\varphi^{(k+1)}\varphi_{j}(w)\| &\text{by \eqref{Eq:j-notation}} \\
&= \ell(c) = \ell(\pi_{\{j\}}(c)) + \ell(\widehat{\pi}_{\{j\}}(c))\\ 
&\geq 4n + \|\varphi^{(k+1)}(w)\|,
\end{align*}
which proves the first alternative in the statement of the lemma.

\subsubsection*{Case 2: {$\ell(\pi_{\{j\}}(c))< 4n$.}}
Recall that if $t_1=x_0$ then $j$ is odd and if $t_1=x_1$ then $j$ is even.
Since $\varphi_j(t_i) = t_iT_j$ or $T_j^{-1}t_i$, it follows that in both cases
$$
\varphi_j(w) = \varphi_j(t_1\cdots t_m) = t_1T'_1t_2T'_2 \cdots t_mT'_m,
$$
where $T'_i = T_j^{\pm 1}$ or $T'_i=1$, and $T'_m=T_j$.
Consequently,
$$
\varphi^{(k+1)}(\varphi_j(w)) = \varphi^{(k+1)}(t_1)\varphi^{(k+1)}(T'_1) \cdots 
\varphi^{(k+1)}(t_m)\varphi^{(k+1)}(T'_m).
$$
Since $T'_m=T_j$ is a subsequence in $\varphi^{(k+1)}(T'_m)$ and since 
$T_j=(x_jy_j)^n(x_j^{-1}y_j^{-1})^n$
is a reduced  word of length $4n$, the cancellation sequence $c$ must leave at least
one letter $x_j^{\pm 1}$ or $y_j^{\pm 1}$ in the subsequence $T'_m$ of $\varphi^{(k+1)}(T'_m)$.
Of these letters of $T'_m$ which do not belong to $c$, let $t$ denote the one
furthest to the right. The folding $\C F$ must pair $t$ with some letter $s$
of $\varphi^{(k+1)}(\varphi_j(w))$. Since $s=t^{-1}$ and since the letters
$x_j^{\pm 1}$ and $y_j^{\pm 1}$ in $\varphi^{(k+1)}(\varphi_j(w))$ only occur
in the subsequences $T'_i$, it follows that $s$ is a letter in $T_i'$ for
some $1\leq i\leq m$. 
Also $s$ must appear to the left of $t$ by the choice of $t$.

Let $v_0$ be the subword of $\varphi_{j}(w)$ between the letters 
$s$ and $t$, inclusive.
Let $v_L$ be the subword of $\varphi_{j}(w)$ consisting of all letters to the left of
$s$ and $v_R$, letters to the right of~$t$, so that $v_Lv_0v_R=\varphi_{j}(w)$.
Applying Lemma \ref{L:fik-anti-triangle} to $\varphi_j(w)$ and $\varphi^{(k+1)}(\varphi_j(w))$
and $(c,\C F)$ we obtain elements $w_L=\rho_{\infty}(v_L)$, $w_0=\rho_{\infty}(v_0)$ and
$w_R=\rho_{\infty}(v_R)$ in $H = \langle x_0,x_1 \rangle$ such that
\begin{equation}
\|\varphi^{(k+1)}(\varphi_{j}(w))\|
\geq \|\varphi^{(k+1)}(w_0)\|+ \|\varphi^{(k+1)}(w_Lw_R)\|.
\label{Eq:lemma-applied}
\end{equation}
Recall from \eqref{eqn:phi_k(T_j)} that $\rho_{\infty}(T_j) = 1$.
Then
\begin{align*}
w_Lw_0w_R &= \rho_{\infty}(v_Lv_0v_R)\\ 
&= \rho_{\infty}(\varphi_{j}(w))\\
&=\rho_{\infty}(t_1 T'_1 t_2 T'_2\cdots t_mT'_{m})\\
&=\rho_{\infty}(t_1\cdots t_m)\\
&=t_1\cdots t_m=w
\end{align*}
It follows from
\eqref{Eq:j-notation} and\eqref{Eq:lemma-applied} that
$$
\|\varphi^{(k-1)}(w)\| \geq
\|\varphi^{(k+1)}(\varphi_{j}(w))\|
\geq \|\varphi^{(k+1)}(w_0)\|+ \|\varphi^{(k+1)}(w_Lw_R)\|.
$$
It remains to prove that $w_0\neq 1$ and $w_Lw_R\neq 1$.
There are two cases depending on the position of the letter $s$ in
$\varphi_{j}(w) = t_1T_1' \cdots t_mT_j$.

\subsubsection*{Case 2.1: {\it The letter $s$ is in the last factor $T'_m$.}}

Recall that the reduced word $T'_m=T_j = (x_jy_j)^n(x_j^{-1}y_j^{-1})^n$
is of even length. Since $s=t^{-1}$ we
get that the subword $v_0$ of $T_j$ is of odd length and therefore so is 
$w_0=\rho_{\infty}(v_0)$ because $\rho_{\infty}$ maps generators to 
generators. 
It follows that $w_0 \neq 1$.

Let $T_j- v_0$ denote the word obtained by removing $v_0$ from $T_j$.
Then
$\rho_{\infty}(T_j- v_0) = x_{j\,{\rm mod}\,2}^p$, where $p$ is an odd integer. 
It follows that
$$
w_Lw_R = \rho_{\infty}(v_Lv_R) = \rho_{\infty}(t_1T_1'\cdots t_m \cdot (T_j-v_0))
=t_1\cdots t_m x_{j\,{\rm mod}\,2}^p.
$$
Since $t_1\cdots t_m$ is reduced and $t_m=x_{(j+1)\,{\rm mod}\,2}$, the word 
$w_Lw_R=t_1\cdots t_m x_{j\,{\rm mod}\,2}^p$ is also reduced and non-trivial,
hence $w_Lw_R\neq 1$ in $H$.

\subsubsection*{Case 2.2: {\it The letter $s$ is in $T'_i$, where $1\leq i\leq m-1$.}}
Write $T'_i=T^L_i\, s\, T^R_i$ and $T'_m=T^L_m\, t\, T^R_m$.
Then
\begin{align*}
\varphi_{j}(w) 
&= 
t_1T_1'\cdots t_i T_i't_{i+1} 
\cdots t_m  T_m'  
\\
&= 
t_1T_1'\cdots t_iT^L_i\, s\, T^R_i t_{i+1} \cdots t_m T^L_m\, t\, T^R_m.
\end{align*}
By construction 
\begin{align*}
v_L &= t_1T'_1\cdots t_iT^L_i\\
v_0 &= sT^R_it_{i+1}\cdots t_m T^L_mt\\
v_R &= T^R_m.
\end{align*}
Since $\rho_\infty(T_j)=1$ by \eqref{eqn:phi_k(T_j)}, it follows that
\begin{align*}
w_Lw_R &= \rho_{\infty}(v_Lv_R) =\rho_{\infty}(t_1T'_1\cdots t_i T^L_iT^R_m)\\
&=t_1\cdots t_i  \cdot \rho_{\infty}(T^L_iT^R_m) \\
&=t_1\cdots t_i \cdot x_{j\,{\rm mod}\,2}^p.
\end{align*}
Then $w_Lw_R$ is conjugate to $x_{j\,{\rm mod}\,2}^p \cdot t_1\cdots t_i$.
The latter word is reduced since $t_1=x_{(j+1)\,{\rm mod}\,2}$, and is  non-trivial since $i \geq 1$.
It follows that $w_Lw_R\neq 1$ in $H$. 
Similarly,
$$
w_0 = 
\rho_{\infty}(s\, T^R_it_{i+1}\cdots T_{m-1}t_mT^L_m\, t) 
= \rho_\infty(s) \cdot x_{j\,{\rm mod}\,2}^pt_{i+1}\cdots t_m x_{j\,{\rm mod}\,2}^q \cdot \rho_\infty(t)
$$
for some $p,q\in \mathbb Z$. 
Since $t=s^{-1}$ this word is conjugate to
$t_{i+1}\cdots t_m x_{j\,{\rm mod}\,2}^{p+q}$ which is reduced and non-trivial
since $i+1 \leq m$ and $t_m = x_{(j+1)\,{\rm mod}\,2}$. This implies that $w_0\neq 1$ and finishes
the proof of the lemma.
\end{proof}

\paragraph{The quadratic lower bound in \eqref{Eq:norm-gn}.}

\begin{proposition}\label{P:final}
Let $w\in H=\langle x_0,x_1\rangle$ be a word of odd length.
Then for any $0\leq k\leq n-1$
$$
\|\varphi^{(2n-2k)}(w)\| \geq k^2+k+1.
$$
In particular,
$$
\|g_n\| = \|\varphi^{(2)}(x_0)\| \geq n^2-n+1,
$$
which proves the lower bound in \eqref{Eq:norm-gn}.
\end{proposition}
\begin{proof}
Use induction on $k$. 
For $k=0$ it is true, since $\varphi^{(2n)}={\rm Id}$ and $\|w\|\geq 1$.
Assume the statement is true for $0 \leq k\leq n-2$ and we will prove it for $k+1$.
Since the norm is conjugation invariant we may replace $w$ with a cyclically reduced word representing the same element in $H$, which is necessarily of odd length.
Since $w$ is cyclically reduced and since  $\ell(w)$ is odd, a cyclic permutation of $w$ must start and end with the same letter. 
So we may assume that $w$ starts and ends with the same letter.
By applying Lemma \ref{L:intricate} we get that either
\begin{align*}
\|\varphi^{(2n-2k-2)}(w)\|
&\geq 4n + \|\varphi^{(2n-2k)}(w)\|\\
&\geq 4n + k^2+k+1 &\text{by induction hypothesis}\\
&\geq 4(k+2) + k^2+k+1 &\text{because $k\leq n-2$}\\
&\geq (k+1)^2 + (k+1) + 1,
\end{align*}
which completes the induction, or there exist $w_0,w_L$ and $w_R$ in $H$ such that $w_0, w_Lw_R \neq 1$ and 
$w=w_Lw_0w_R$
and
$$
\|\varphi^{(2n-2k-2)}(w)\|\geq \|\varphi^{(2n-2k)}(w_0)\| + \|\varphi^{(2n-2k)}(w_Lw_R)\|.
$$
Since $\ell(w)$ is odd then one of $\ell(w_0)$ and $\ell(w_Lw_R)$ is odd and the other
is even. The induction hypothesis together with Proposition \ref{P:fiw-linear} imply that
$$
\|\varphi^{(2n-2k-2)}(w)\|\geq (k^2+k+1) + (2k+2) = (k+1)^2+(k+1)+1,
$$
which completes the induction step and finishes the proof.
\end{proof}

\section{The construction in $\B F(2)$}\label{SS:F2}
\begin{lemma}\label{L:transplant}
Let $\Sigma$ be a countable set.
For any integer $m\geq 2$ there exists a homomorphism $\Psi_m\colon \B F(\Sigma) \to \B F(2)$
such that
$$
\|\Psi_m(g)\| = \|g\|
$$
for every $g\in \B F(\Sigma)$ represented by a word of length at most $m$.
\end{lemma}
\begin{proof}
Enumerate $\Sigma=\{s_1,s_2,\dots\}$. 
Let $a,b$ denote the generators of $\B F(2)$.
Define $\Psi_m\colon \B F(\Sigma) \to \B F(2)$
by
\begin{equation}
\Psi_m(s_i) = b^{-m^i} a b^{m^i}.
\label{Eq:transplant}
\end{equation}
Observe that $\|\Psi_m(s_i)\| = 1$ and hence $\Psi_m$ is Lipschitz with constant
$1$. That is, $\|\Psi_m(g)\|\leq \|g\|$ for every $g\in \B F(\Sigma)$. It remains
to prove the opposite inequality if $|g| \leq m$.

Suppose that there exists some $g \in \B F(\Sigma)$ represented by a word of length at most $m$ such that $\|\Psi_m(g)\|<\|g\|$.
Among such elements choose $g$ with minimal $|g|$.
Notice that $g$ must be non-trivial. 
Since $|g|\leq m$, we also have $\|g\|\leq m$. 
The reduced word representing $g$ has the form
\begin{equation}\label{Eq:g syllables}
g=s_{i_1}^{n_1}s_{i_2}^{n_2}\cdots s_{i_k}^{n_k},
\end{equation}
where $k \geq 1$, $0\neq n_j\in \B Z$ for $1\leq j\leq k$, $s_{i_j}\neq s_{i_{j+1}}$ for $1\leq j\leq k-1$, and $\sum_{j=1}^k |n_j|=|g| \leq m$.
Clearly,  $|n_j|\leq m$ for all $1 \leq j \leq k$.
The power $s_{i_j}^{n_j}$ is called a {\em syllable}. 
By construction,
\begin{align}
\Psi_m(g) 
&= \prod_{j=1}^k\ b^{-m^{i_j}}\cdot a^{n_j}\cdot b^{m^{i_j}}\nonumber\\
&=b^{\ell_0} a^{n_1} b^{\ell_1} a^{n_2} b^{\ell_2} \, \cdots \, b^{\ell_{k-1}}a^{n_k}b^{\ell_k}\label{Eq:psig},
\end{align}
where $\ell_0=-m^{i_1}$ and $\ell_j=m^{i_{j}}-m^{i_{j+1}}$ for $j=1,\ldots,k-1$, and $\ell_k=m^{i_k}$.
Notice that 
\[
\sum_{j=0}^k {\ell_j} = -m^{i_1} \, + \left(\sum_{j=1}^{k-1} m^{i_{j}}-m^{i_{j+1}}\right) \, + m^{i_k} =0.
\]
Since $s_{i_j} \neq s_{i_{j+1}}$, i.e. $i_j \neq i_{j+1}$ for all $1 \leq j \leq k-1$, and since $m \geq 2$, it follows that $|\ell_j| \geq m$ for all $0 \leq j \leq k$ because $|\ell_0| = m^{i_1}$ and $\ell_k = m^{i_k}$ and for $1 \leq j \leq k-1$
\[
| \ell_j| = |m^{i_j}-m^{i_{j+1}}| = m^{\min\{ i_j,i_{j+1}\}} \cdot  | m^{|i_j-i_{j+1}|}-1| \geq m^{\min\{ i_j,i_{j+1}\}} \geq m.
\]
In particular, the presentation of $\Psi_m(g)\in \B F(2)$ in \eqref{Eq:psig} is a reduced non-trivial word with all syllables
involving $b$ of length $ \geq m$. Since $k\geq 1$, the  expression
also has a syllable involving $a$.

Let $c$ be the minimal cancellation sequence for the reduced word \eqref{Eq:psig}.
It must contain at least one syllable. If it contains a syllable $b^{\ell_j}$ for
some $j=0,1,\ldots,k$, then it must contain at least another occurrence of $b^{\pm 1}$ since $|\ell_j| \neq 0$ and $\sum \ell_j = 0$.
Therefore 
$$
\|\Psi_m(g)\| = |c| \geq |\ell_j| + 1 > m  \geq  |g| \geq \|g\|, 
$$
which contradicts the fact that $\Psi_m$ is $1$-Lipschitz. So $c$ must contain
a syllable involving $a$, say $a^{n_{k'}}$ for some $k'\in\{1,\ldots,k\}$.
Let $c'$ be $c$ with the syllable $a^{n_{k'}}$ removed and consider $g' \in \B F(\Sigma)$ defined by 
\begin{equation}\label{Eq:g' syllable}
g' = s_{i_1}^{n_1}s_{i_2}^{n_2}\cdots \widehat{s_{i_{k'}}^{n_{k'}}}\cdots s_{i_k}^{n_k}.
\end{equation}
That is, $g'$ is obtained from $g$ by removing the syllable $s_{i_{k'}}^{n_{k'}}$ in the presentation \eqref{Eq:g syllables} of $g$.
The word in \eqref{Eq:g' syllable} representing $g'$ {\em need not be reduced}.
By the triangle inequality and since $g$ is the product of $g'$ and a conjugate of $s_{i_{k'}}^{n_{k'}}$
\begin{equation}
\|g\| \leq \|g'\|+\| s_{i_{k'}}^{n_{k'}} \| = \|g'\|+|n_{k'}|
\label{Eq:triangle-gg'}
\end{equation}
Since $|g'|<|g|$, by  minimality of $|g|$, 
\begin{equation}
\|\Psi_m(g')\| = \|g'\|.
\label{Eq:g'}
\end{equation}
Moreover, $c'$ is a cancellation sequence in the word
$$
\prod_{j=1}^{k'-1}\ b^{-m^{i_j}}\cdot a^{n_j}\cdot b^{m^{i_j}}\cdot
\prod_{j=k'+1}^{k}\ b^{-m^{i_j}}\cdot a^{n_j}\cdot b^{m^{i_j}}
$$
representing $\Psi_m(g')$. Combining the above observations we get the
following chain of inequalities which contradicts the initial assumption
that $\|\Psi_m(g)\|<\|g\|$.
\begin{align*}
\|\Psi_m(g)\| &= |c|\\ 
&= |c'|+|n_{k'}| \\
& \geq \|\Psi_m(g')\|+|n_{k'}| \\
&= \|g'\|+|n_{k'}| &\text{by \eqref{Eq:g'}}\\
&\geq \|g\| & \text{by \eqref{Eq:triangle-gg'}}
\end{align*}
This completes the proof.
\end{proof}

Let $n\geq 2$ and let $g_n\in \B F_{\infty}$ be the elements defined in \eqref{Eq:gn}.
Set 
\[
m_n=2(4n+1)^{2n+2}.
\] 
By \eqref{Eq:ell g_n} $g_n$ is represented by a word of length $\leq (4n+1)^{2n+2}$ so $g_n^2$ is represented by a word of length $\leq m_n$. 
We get the following consequence of Lemma \ref{L:transplant} and Theorem \ref{T:bounds}.

\begin{corollary}\label{C:F2}
In the above notation we have the following inequalities:
\begin{align*}
n^2-n+1 &\leq \|\Psi_{m_n}(g_n)\| = \|g_n\| \leq 4n^2-8n+5\\
2n &\leq \|\Psi_{m_n}(g_n^2)\|= \|g_n^2\| \leq 8n-4.
\end{align*}
\qed
\end{corollary}

\section{Proof of Theorem \ref{T:main}}

Recall that the asymptotic cone ${\rm Cone}_{\omega}(G,d)$ of
a group $G$ with a bi-invariant metric $d$ is equipped
with the bi-invariant metric given by
$$
d_{\omega}([g_n],[h_n]) = \lim_{\omega}\frac{d(g_n,h_n)}{n}.
$$
The associated norm is given by $\|[g_n]\|_{\omega}=\lim_{\omega}\frac{\|g_n\|}{n}$.

Recall that a function $f \colon \mathbb N \to \mathbb R$ is asymptotically $\Theta(n)$ if it is both $O(n)$ and $\Omega(n)$, i.e.,
\[
0 < \liminf_{n \to \infty} \frac{f(n)}{n} \leq \limsup_{n \to \infty} \frac{f(n)}{n} < \infty.
\]

\begin{proof}[Proof of Theorem \ref{T:main}]
For $n\geq 1$, set
\[
r_n=\max\{2,\lfloor\sqrt n\rfloor\},\qquad h_n=g_{r_n}.
\]
Then $r_n\sim\sqrt n$, and Theorem~\ref{T:bounds} gives
$\|h_n\|=\Theta(n)$ and $\|h_n^2\|=O(\sqrt n)$, hence $\|h_n^2\|=o(n)$.
Consequently, it follows from Corollary \ref{C:F2} that
the elements $f_n=\Psi_{m_{r_n}}(h_n)\in \B F(2)$ satisfy
$\|f_n\| =\Theta(n)$ and $\|f_n^2\| =o(n)$. 
This implies
that $[f_n]\neq 1$ in the asymptotic cone ${\rm Cone}_{\omega}(\B F(2),d)$
and $[f_n]^2=[f^2_n]=1$. That is, $[f_n]$ is non-trivial of order $2$.

Let $S$ be any set of cardinality at least $2$.
Since any injective map $\{a,b\}\to S$ defines an injective homomorphism 
$\B F(2) \to \B F(S)$
which is an isometric embedding
with respect to the bi-invariant metrics associated
with $\{a,b\}$ and $S$, respectively, it induces an injective homomorphism between
the asymptotic cones which is also an isometric embedding. Consequently,
the image of $[f_n]$ in ${\rm Cone}_{\omega}(\B F(S),d_{\overline{S}})$
is a non-trivial element of order $2$.
\end{proof}

\end{document}